\documentclass{article}%
\usepackage{amsmath, amssymb}%
\usepackage{amsfonts}%
\usepackage{amssymb}%
\usepackage{graphicx}
\newtheorem{theorem}{Theorem}[section]
\newtheorem{lemma}[theorem]{Lemma}
\newtheorem{proposition}[theorem]{Proposition}
\newtheorem{definition}[theorem]{Definition}
\newcommand{\qed}{\rule{1mm}{3mm}}     
\newcommand{\N}{{\mathbb N}}

\newenvironment{proof}{\vspace*{\parsep}\noindent {\bf Proof:}}{\qed\\[1em]}

\begin{document}

\title {On the Failure of BD-${\mathbb N}$ and BD, and an
Application to the Anti-Specker Property}

\author{Robert S. Lubarsky \\ Dept. of Mathematical Sciences
\\ Florida Atlantic University \\ Boca Raton, FL 33431 \\
Robert.Lubarsky@alum.mit.edu} \maketitle
\begin{abstract}
We give the natural topological model for $\neg$BD-${\mathbb N}$,
and use it to show that the closure of spaces with the
anti-Specker property under product does not imply BD-${\mathbb
N}$. Also, the natural topological model for $\neg$BD is
presented. Finally, for some of the realizability models known
indirectly to falsify BD-$\mathbb{N}$, it is brought out in detail
how BD-$\mathbb N$ fails.
\\ {\bf keywords:} anti-Specker, BD, BD-$\mathbb{N}$, realizability,
topological models\\{\bf AMS 2010 MSC:} 03F60, 26E40
\end{abstract}

\section{Introduction}
In recent years much attention has been paid to subtle
foundational principles of constructive analysis. (For background
in constructive analysis, see \cite {B,BB}.) These are principles
that hold in the major traditions of mathematics, such as
Brouwer's intuitionism, Russian constructivism, and classical
mathematics, yet do not follow from ZF-style axioms on the basis
of constructive logic, such as IZF. The principles of this
character identified so far are Weak Markov's Principle \cite{BR,
I91, TvD}, a version of Baire's Theorem \cite{IS}, and, most
importantly for our purposes, the boundedness principles BD and
BD-$\mathbb{N}$.

By way of explaining these boundedness principles, a sequence
($a_n$) of natural numbers is {\bf pseudo-bounded} if $\lim_{n
\rightarrow \infty} a_n/n = 0$. A set of natural numbers is {\bf
pseudo-bounded} if every sequence of its members is
pseudo-bounded. Examples are bounded sets. Equivalently, as
observed in \cite {IY}, a set is pseudo-bounded if every sequence
$(a_n)$ of its members is eventually bounded by the identity
function: for $n$ large enough, $a_n < n$ (or $a_n \leq n$). (To
see this equivalence, consider large enough intervals within
$(a_n)$.) This latter formulation is often easier to work with.

BD is the assertion that every pseudo-bounded set of natural
numbers is bounded. BD-$\mathbb{N}$ is that every countable
pseudo-bounded set of naturals is bounded. They were first
identified during Ishihara's analysis of continuity \cite {I92},
as they are equivalent to sequentially continuous functions on
certain metric spaces being $\epsilon-\delta$ continuous. Since
their identification, they have become central tools in the
foundations of constructive analysis, especially the latter \cite
{Br09, Br0, BISV, I01, IY}.

It can be hard to imagine how BD or BD-$\mathbb{N}$ could fail,
which is likely a cause or effect of their being true in most
systems. That notwithstanding, in order to understand them better,
it is useful to see when they are false. Trivially BD implies
BD-$\mathbb N$, so when discussing their failure we will usually
restrict attention to the weaker of the two, the one more
difficult to falsify, BD-$\mathbb N$. It turns out that the first
models falsifying it did so unwittingly. BD-$\mathbb N$ is new,
and continuity is old. So models violating commonly accepted
continuity principles were developed long before BD-$\mathbb N$
was even identified. It was only later that people looked back
and realized that the only way that those continuity properties
could fail was through the failure of BD-$\mathbb N$.

There are several shortcomings to this state of affairs, which the
current work is intended to address. One is that these first
models seem somewhat ad hoc for this purpose, falsifying
BD-$\mathbb N$ almost by accident. In contrast, the topological
models presented here seem to be the natural models. (For
discussion about naturality, see also the questions at the end.
For background on topological models, see \cite {G1,G2}.) That is,
to violate BD-$\mathbb N$, you'd need a sequence which is sort-of
bounded while also sort-of unbounded. Without thinking much about
it, you might well guess that either a generic bounded sequence or
a generic unbounded sequence would do the trick. This turns out to
be exactly right, as we will see. Similarly, to violate BD, you'd
need a set which is simultaneously bounded and unbounded, after a
fashion. The first guess is again a generic set, either bounded or
unbounded; again, this does it.

A second shortcoming of the prior state of knowledge is that the
way we know BD-$\mathbb N$ to fail in these first models is
indirect: BD-$\mathbb N$ plus other foundational axioms imply some
continuity principle; said continuity principle fails; check that
the other axioms hold; hence BD-$\mathbb N$ fails. We are left
with the unsatisfying feeling of not really knowing just why
BD-$\mathbb N$ fails. What is the pseudo-bounded sequence which is
not bounded? Or is there something else going on? This is also
addressed later, when for some of these models the indirect
argument above is unraveled to reveal just how BD-$\mathbb N$
fails.

Finally, the first proofs of the independence of BD-$\mathbb N$
were in Lietz's thesis \cite {L}, which also contains the bulk of
the known models violating BD-$\mathbb N$. Although these models
are all ultimately realizability models, the presentation is
category-theoretic through and through, and so difficult or even
inaccessible for non-category theorists to understand. The
presentation of these models given here is purely in terms of
realizability for arithmetic and analysis.

Returning to the alleged naturality of the topological models, a
good test for naturality is whether the model, while violating BD
or BD-$\mathbb N$, violates as little else as possible. That is,
it should prove independence results around BD or BD-$\mathbb{N}$.
For instance, Doug Bridges \cite{Br} has shown that, under
BD-$\mathbb{N}$, anti-Specker spaces (see below) are closed under
products, and asked whether the converse holds. If there is a
canonical model, whatever that might mean, falsifying
BD-$\mathbb{N}$, that would be the first place to look for this
question. Such a model would be the gentlest possible extension of
a classical model making BD-$\mathbb{N}$ false. If the failure of
BD-$\mathbb{N}$ did not imply that anti-Specker spaces are so
closed, then in the canonical model the anti-Specker spaces would
retain this closure. This is in fact just what is shown in section
3, that in the first model of section 2 the anti-Specker spaces
are closed under Cartesian product. Hence such closure does not
imply BD-$\mathbb{N}$.

We hope that the models presented here will make investigations
around BD and BD-$\mathbb{N}$ easier, as well as promote interest
in topological models more generally.

A word about the meta-theory is in order. This article, like
almost all in mathematics, is intended to be within classical
mathematics. This would be unremarkable, except that it is at the
same time an article about constructive mathematics. Just as we
are more upset when religious leaders engage in even common sexual
transgressions than when lay people do, because we hold them to a
higher standard, so are people who discuss constructive
mathematics questioned about their use of Excluded Middle when
almost everyone else does so without even noticing it. Given the
likely readership, though, it would be a service to the reader to
bring out the not fully constructive principles in this work when
they are used. This is done, for instance, in the proof of \ref
{main lemma}, which is phrased as a proof by contradiction, and in
reality uses only the Fan Theorem, as is observed there. The
bottom line caveat to the reader, though, is that this work has
not been vetted to be purely constructive, and may contain
unnoticed applications of Excluded Middle.

\section {Topological models}

As discussed above, we are looking for a sequence (or, in the case
of BD, set) which is bounded in a way, yet also unbounded in
another way. The obvious guess is to take either a generic bounded
or a generic unbounded sequence. That is, the topological space
would be the space of bounded sequences, or unbounded sequences.
The topological model over a space introduces a generic elements
of that space, so in this case we'd have a generic bounded, or
unbounded, sequence. Not surprisingly, this works for both
BD-$\mathbb N$ and BD.

\subsection{The natural model for $\neg$BD-$\mathbb N$} Let the
points in the topological space $T$ be the functions $f$ from
$\omega$ to $\omega$ with finite range, that is, enumerations of
non-empty finite sets. A basic open set $p$ is (either $\emptyset$
or) given by an unbounded sequence $g_p$ of integers, with a
designated integer ${\rm stem}(p)$, beyond which $g_p$ is
non-decreasing. Furthermore, $g_p({\rm stem} (p)) \geq
\max\{g_p(i) \mid i < {\rm stem}(p)\}$. A function $f$ is in $p$
if $f(n) = g_p(n)$ for $n < {\rm stem}(p)$ and $f(n) \leq g_p(n)$
otherwise. \footnote {The referee has observed that this topology
is the join of two other topologies, namely the product topology
on Baire space as a product of copies of $\mathbb N$, and the
``lower" topology for non-decreasing maps, with basic open sets
$\{f \in T \mid \forall n f(n) \leq g(n)\}$ for $g$
non-decreasing. They pointed out that $T$ is therefore
zero-dimensional and Hausdorff, and hence sober.} Notice that $p
\cap q$ is either empty (if $g_p$ and $g_q$ through their stems
are incompatible) or is given by taking the larger of the two
stems, the function up to that stem from the condition with the
larger stem, and the pointwise minimum beyond that. Hence these
open sets do form a basis.

Let $M$ be the full topological model built over $T$. For those
with less experience with such models, perhaps the easiest way to
think of this is that we are forcing over the partial order of
open subsets of $T$ without modding out by double negation. (The
following discussion applies to any topological space $T$.) Sets
in $M$ are named by sets in the ground model $V$. Inductively in
$V$, $M_0 = \emptyset, M_\lambda = \bigcup_{\alpha < \lambda}
M_\alpha,$ and $M_{\alpha + 1}$ is the power set of $M_{\alpha}
\times O(T)$, where $O(T)$ is the set of open subsets of $T$.
Clearly all the action is happening at the successor stage. A
member $\sigma$ of $M_{\alpha + 1}$, called a term, is of the form
$\{ \langle \sigma_j, r_j \rangle \mid j \in J \}$, where
$\sigma_j$ is (inductively) a term, $r_j$ an open set, and $J$ an
index set. Thinking of $r_j$ as a forcing condition, $r_j$ forces
$\sigma_j$ to be a member of $\sigma$; alternatively, $r_j$ is the
truth value of the sentence $``\sigma_j \in \sigma"$. Just as with
classical forcing, the ground model embeds canonically into $M$:
${\hat x} = \{ \langle {\hat y}, T \rangle \mid y \in x \}$
provides a canonical $M$-name for each $x \in V$. Where this
differs from classical forcing is that classically for a sentence
$\phi$ to be true in the generic extension its truth value must be
a dense open subset of $T$, whereas here the truth value would
have to be all of $T$. Hence it might be misleading to use the
word ``generic" in this context, since density plays little of a
role here. Still, we are going to do so, because we have no better
name for what is in forcing called ``the generic $G$", given by
the term $\{ \langle {\hat r}, r \rangle \mid r \in O(T) \}$. To
try to convey some intuition, $G$ looks like a generic point of
$T$, which is in whatever open set you happen to be looking at at
the moment. (For more background on topological models, see \cite
{G1,G2}. What is there called the topological model is here
described as the full model, to distinguish it from other possible
models. The fullness consists of it containing all possible terms.
The discussion in the addendum to \cite {RSL10} might also be of
use here.)

Particularizing to the space $T$ at hand, $G$ is forced to be a
function with domain the natural numbers, themselves the canonical
image of the ground model natural numbers, and can be
characterized by the relation $p \Vdash ``G(n) = x"$ iff $n < {\rm
stem}(p)$ and $g_p(n) = x$. Our interest is really in the range of
$G$. The reason we force to get $G$ as a function, and not just
the range, is that BD-$\mathbb N$ refers to countable sets, that
is, the ranges of functions with domain the natural numbers. We
would have to do something at some point to make the set being
built countable, so it just seemed easier to take care of this up
front. The primary result of this note is

\begin{theorem} $T \Vdash rng(G)$ is countable and pseudo-bounded,
yet not bounded. Also, $T \Vdash$ Dependent Choice (DC).
\end{theorem}

A major reason we're interested in showing that the model
validates DC is that it makes the failure of BD-$\mathbb N$ that
much more striking. One might consider arguing for BD-$\mathbb N$
by, given $f:\mathbb N \rightarrow \mathbb N$, trying to build a
sequence with $a_0 = f(0)$ and $a_{n+1} = \min \{x \mid x \in
rng(f), x > a_n\}$; if this construction succeeds, then $f$ is not
pseudo-bounded, and if it doesn't then $f$ looks to be bounded.
This will not work, if for no other reason than that this minimum
might not exist. With DC, that last objection loses its validity:
as long as $\{x \mid x \in rng(f), x
> a_n\}$ is inhabited, a value could be chosen for $a_{n+1}$, and
with DC these values could be strung together into a total
sequence. Hence DC makes it even harder for BD-$\mathbb N$ to
fail. Another reason to want DC is that it is useful for
developing analysis. So this model then provides a nice testing
ground or example for the development of analysis in the absence
of BD-$\mathbb N$.

\begin {proof}
It is immediate that $rng(G)$ is countable, as $G$ is total: given
$n$, $\{p \mid {\rm stem}(g_p) > n\}$ covers $T$, and each such
$p$ forces $n \in dom(G)$. It is almost as immediate that $rng(G)$
is not bounded: given a potential bound $B \in \mathbb{N}$, since
$g_p$ is unbounded, $p$ can be extended (using standard
terminology from forcing here, the meaning being ``shrunk" as a
set) to an open set forcing $B \in rng(G)$; hence no $p$ can force
that any particular natural is a bound, i.e. nothing forces that
$rng(G)$ is bounded, so $T$ forces that $rng(G)$ is unbounded.

The work is in showing pseudo-boundedness. The primary lemma for
that is:

\begin{lemma} \label {main lemma}
Let $p$ be an open set forcing $``t \in rng(G)"$, and $I$ an
integer such that $\max_{n<{\rm stem}(p)} g_p(n) \leq I \leq
g_p({\rm stem}(p)).$ Then there is a q extending p with the same
stem and $g_q({\rm stem}(q)) \geq I$ forcing $``t \leq {\hat I}"$.
\end{lemma}
\begin{proof}
If $r$ is an open set, for $i \leq g_r({\rm stem}(r))$, let $r_i
\subseteq r$ be such that ${\rm stem}(r_i) = {\rm stem}(r) + 1$,
$g_{r_i}({\rm stem}(r)) = i$, and for $n \not = {\rm stem}(r), \;
g_{r_i}(n) = g_r(n)$. Notice that $\bigcup_i r_i = r$.

Fix $I$. Say that $p' \subseteq p$ is a \underline{candidate} if
$\max_{n<{\rm stem}(p')} g_{p'}(n) \leq I \leq g_{p'}({\rm
stem}(p')).$ If $p'$ is a candidate, say that $q' \subseteq p'$ is
\underline{good} if $q'$ satisfies the conclusion of the lemma as
applied to $p'$. We want to show that $p$ itself has a good
extension. Notice that if $i \leq I$ and $p'$ is a candidate, so
is $p'_i$.

Suppose that each $p_i \; (i \leq I)$ had a good extension, say
$q^i$. Then so would $p$, as follows. Let ${\rm stem}(q) = {\rm
stem}(p)$. For $n < {\rm stem}(q)$, let $g_q(n)$ be the common
value $g_p(n)$; let $g_q({\rm stem}(q))$ be $I$; for all other
$n$, let $g_q(n)$ be $\min_i g_{q^i}(n)$. As described above, $q$
is covered by the $q_i$s, and each $q_i$ is a subset of $q^i$, so
$q \Vdash ``t \leq {\hat I}"$.

This means that if $p^0 := p$ does not have a good extension,
neither does some $p^1 := p_i$. Continuing inductively, let
$p^{n+1}$ be $p^n_i$, where $i \leq I$ and $p^n_i$ does not have a
good extension. The initial parts of the $p^n$s cohere to form a
function $f \in p$ with range (a subset of) \{0, 1, ... , $I$\}.
Let $r \subseteq p$ be a neighborhood of $f$ forcing a particular
value $J$ for $t$. We have that $J \leq I$, as follows. Since $t$
is forced to be in the range of $G$, extend $r$ to a neighborhood
$s$ of $f$ forcing $t = G(m)$ for some natural number $m$. Without
loss of generality, ${\rm stem}(s)$ can be taken to be larger than
$m$. By the choice of $f$, $G(m) \leq I$, as claimed.

Eventually the values of $g_s$ are all $\geq I$, so, if need be,
shrink $s$ by extending the initial part consistently with $f$
until $g_s({\rm stem}(s)) \geq I$. This is a good extension of
some $p^n$, which is a contradiction.

Alternatively, this argument could be done on the basis of the Fan
Theorem instead of a proof by contradiction, since we are looking
at trees with $I$-much branching.
\end{proof}

In the end, we will need a more general version of the preceding.

\begin{lemma}
Let $p$ be an open set forcing $``t \in rng(G)"$, and $M \geq {\rm
stem}(p), I$ an integer such that $\max_{n<M} g_p(n) \leq I \leq
g_p(M).$ Then there is a q extending p such that for $n < M \;
g_q(n) = g_p(n), {\rm stem}(q) = {\rm stem}(p),$ and $g_q(M) \geq
I$, forcing $``t \leq {\hat I}"$.
\end{lemma}

Notice that the previous lemma is the special case of the current
one in which $M = {\rm stem}(p)$.

\begin{proof}
Using the notation from the beginning of the last proof, $\{ p_i
\mid i \leq g_p({\rm stem}(p)) \}$ is an open cover over $p$.
Similarly, $\{ (p_i)_j \mid j \leq g_p({\rm stem}(p)+1) \}$ is an
open cover of $p_i$, so that the collection of $(p_i)_j$s as both
$i$ and $j$ vary is an open cover of $p$. Continuing this
procedure for $M - {\rm stem}(p)$ many steps, we have an open
cover $q_k$ ($k<K$) of $p$ such that ${\rm stem}(q_k) = M$.
Applying the previous lemma to each $q_k$ produces a collection
$q_k'$ such that each one forces $``t \leq {\hat I}"$. The union
$\bigcup_k q_k'$ can be restricted to a basic open set $q$ as
follows: ${\rm stem}(q) = {\rm stem}(p), g_q(n) = g_p(n)$ for
$n<M$, and for larger $n \; g_q(n) = \min_k(q_k'(n))$. $q$ is as
desired.
\end{proof}

The benefit of this lemma is that it enable one to do fusion
arguments, like in Axiom A forcing \cite {Ba} and in arguments to
get minimal degrees in computability theory, to get the
pseudo-boundedness of G, as well as DC.

\begin{lemma} $T \Vdash rng(G)$ is pseudo-bounded.
\end{lemma}
\begin{proof}
Suppose $p \Vdash a_n$ is a countable sequence through $rng(G)$.
We need to show that $p$ forces that $a_n$ is eventually bounded
beneath $n$. That means that $p$ forces the existence of an index
$N$ beyond which $a_n$ is forced to be less than $n$. Since
forcing existence is local, for any $f \in p$ we need to find a
neighborhood $r$ of $f$ forcing the adequacy of a particular index
$N$.

Fix $f \in p$. Let $N$ be $\sup(rng(f))$. Let $M$ be the smallest
natural such that $g_p(M) > N$. Notice that, since $f \in p, \; M
\geq {\rm stem}(p)$. Apply the previous lemma to $t := a_N$ to
extend $p$ to $q_N$. (Notice that $f \in q_N$, since $g_p$ and
$g_{q_N}$ agree up to $M$, and beyond that $g_{q_N}(n)$ always
bounds $rng(f)$.)

Now apply the previous lemma to $p := q_N, t := a_{N+1}, I :=
N+1,$ and $M$ the least index $n$ such that $g_{q_N}(n) \geq N+1.$
This produces a basic open set $q_{N+1}$ containing $f$ forcing
$a_{N+1} \leq N+1.$

Again, apply the previous lemma to $p := q_{N+1}, t := a_{N+2}, I
:= N+2,$ and $M$ the least index $n$ such that $g_{q_{N+1}}(n)
\geq N+2.$ This produces a basic open set $q_{N+2}$ containing $f$
forcing $a_{N+2} \leq N+2.$

By continuing this process for ever increasing values of $I$, the
function which is the pointwise limit of $g_{q_{N+i}}$ is
unbounded, and so (together with ${\rm stem}(p)$) determines an
open set $\bigcap_n q_n$ containing $f$ that by construction
forces each $a_n$ ($n \geq N$) to be bounded by $n$.
\end{proof}

In order to prove DC, we need appropriate analogues of the two
lemmas from above.

\begin{lemma}
Let $p$ be an open set forcing $``\exists y \; \psi(y)"$, and $I$
an integer such that $I \leq g_p({\rm stem}(p)).$ Then there is a
q extending p with the same stem and $g_q({\rm stem}(q)) \geq I$
forcing $\psi(\sigma)$ for some term $\sigma$.
\end{lemma}

\begin{proof}
This argument is similar to that of the first lemma above. We need
the same notation: $p_i$, for $i \leq g_p({\rm stem}(p))$, is (to
put it informally) the same as $p$ only the value at the ${\rm
stem}(p)^{th}$ place is fixed to be $i$. If each $p_i$ has a good
extension $q^i$, meaning one satisfying the conclusion of the
lemma (when starting from $p_i$), then so does $p$, as follows.
Let $q$ be such that ${\rm stem}(q) = {\rm stem}(p)$, for $n <
{\rm stem}(q)$ $g_q(n)$ is the common value $g_p(n)$, $g_q({\rm
stem}(q))$ is $I$, and for all other $n$ $g_q(n)$ is $\min_i
g_{q^i}(n)$. $q$ is covered by the $q_i$s, each $q_i$ is a subset
of $q^i$, and each $q^i$ forces $\psi(\sigma_i)$. The $\sigma_i$s
can be amalgamated as follows. Recall that a term $\tau$ is a set
of the form $\{ \langle \tau_j, r_j \rangle \mid j \in J \}$,
where $\tau_j$ is (inductively) a term, $r_j$ an open set, and $J$
an index set. The restriction of $\tau$ to some open set $r$,
$\tau \upharpoonright r$, is defined as $\{ \langle \tau_j, r_j
\cap r \rangle \mid j \in J \}$. Intuitively, $\tau
\upharpoonright r$ is empty until you're beneath $r$, at which
point it becomes $\tau$. The amalgamation we want is $\sigma :=
\bigcup_i \sigma_i \upharpoonright q^i$, which roughly stands for
``wait until you know which $q^i$ you're in, then become
$\sigma_i$". This $\sigma$ witnesses that $q$ is a good extension
of $p$.

So if $p^0 := p$ did not have a good extension, neither would some
$p^1 := p_i$, nor some $p^2 := p^1_j$, etc. The $p^n$s cohere, or
converge, to some $f \in p$. By hypothesis, $f$ has some
neighborhood $r$ forcing $\psi(\sigma)$ for some $\sigma$. If need
be, shrink $r$ (to $r'$ say) consistently with $f$ so that
$g_{r'}({\rm stem}(r')) \geq I$. This $r'$ is a good extension of
some $p^n$, which is a contradiction. Again, as with \ref {main
lemma}, instead of a proof by contradiction, the Fan Theorem
suffices.
\end{proof}

\begin{lemma}
Let $p$ be an open set forcing $``\exists y \; \psi(y)"$, and $M
\geq {\rm stem}(p), I$ an integer such that $I \leq g_p(M).$ Then
there is a q extending p such that for $n < M \; g_q(n) = g_p(n),
{\rm stem}(q) = {\rm stem}(p),$ and $g_q(M) \geq I$, forcing
$\psi(\sigma)$ for some term $\sigma$.
\end{lemma}

The proof of this lemma is to the previous proof as the proof of
the lemma before that is to the one before it, and so is left to
the reader.

\begin{lemma} $T \Vdash$ DC.
\end{lemma}

\begin{proof}
Suppose $p \Vdash ``\forall x \; \exists y \; \phi(x,y)"$. (The
argument is unchanged if $x$ and $y$ are restricted to some set.)
Let $a_0$ be given. Using the previous lemma, at stage 0, extend
$p$ to $p_0$, with the same stem and the same value $I$ at ${\rm
stem}(p)$, forcing $\phi(a_0, a_1)$ for some $a_1$. At stage 1,
let $N_1$ be the least index $n$ such that $g_{p_0}(n) > I$, and
extend $p_0$ to $p_1$, with the same stem and same values at all
indices $n \leq N_1$, forcing $\phi(a_1, a_2)$ for some $a_2$.
Continue inductively, at stage $n$ preserving some occurrence of
an integer at least as large as $I+n$ as a function value, so that
$\bigcap_n p_n$ is an open set. By construction, $\bigcap_n p_n
\Vdash \forall n \in \N \; \phi(a_n, a_{n+1})$.
\end{proof}

This sequence of lemmas completes the proof of the main theorem.

\end {proof}

\subsection {The natural model for $\neg$BD}
BD is the assertion that every pseudo-bounded set of natural
numbers is bounded: BD-$\mathbb{N}$ without the assumption of
countability. So BD implies BD-$\mathbb{N}$, and the model above
falsifying BD-$\mathbb{N}$ must also falsify BD. We can do better
than that though. The generic above was not bounded: it is false
that there exists a bound. That's different from being unbounded:
$\forall N \exists i \in A \; i>N$. There are two good reasons
that $G$ above was not unbounded. For one, since $G$ is countable,
if it were unbounded it would not be pseudo-bounded. (Let $a_0$ be
$G(0)$; given $a_n = G(m)$, let $a_{n+1}$ be $G(k)$, where $k$ is
the least integer greater than $m$ such that $G(k) > G(m)$, which
exists by unboundedness. Then $(a_n)$ would witness that $G$ is
not pseudo-bounded.) For another, even if $G$ were not countable,
DC is enough to take an unbounded set and return a witness to
non-pseudo-boundedness. So we would like to find a counter-example
to BD which is unbounded.

More than that, we would like to find a counter-example to a
weakened version of BD. This is based on the following:

\begin {definition} $A \subseteq \mathbb{N}$ is {\bf sequentially
bounded} if every sequence of members of $A$ is bounded.
\end{definition}

Notice that if $A$ is sequentially bounded then $A$ is
pseudo-bounded. The converse does not hold, as $G$ from the
previous theorem illustrates. So the assertion ``if $A$ is
sequentially bounded then $A$ is bounded" differs from BD in that
it has a stronger hypothesis, and so is a weaker assertion. Weaker
assertions are harder to falsify. Hence the goal is to produce a
sequentially bounded, unbounded set. Can this be done?

The answer is a very satisfying yes, satisfying because it so well
complements the previous construction. Recall that it was argued
that the best guess for the last section's counter-example was to
take a generic over either the bounded or the unbounded sequences.
It turned out that the bounded sequences did the trick. The best
guess here would be a generic over either the bounded or unbounded
sets. It turns out that the choice this time is the other one: the
unbounded sets.

So let $T$ be the space of all unbounded sets $X$ of natural
numbers. A basic open set $O$ is given by a pair $\langle P_O,N_O
\rangle$, so called because $P_O$ is a finite set of naturals
giving the positive information and $N_O$ a set of naturals giving
the negative information. $X \in O$ iff $P_O \subseteq X$ and $N_O
\cap X$ is finite. (Hence for the second component we could have
taken instead an equivalence class from the power set of $\mathbb
N$ mod the ideal of finite sets.) For $O$ and $U$ open, $O \cap U$
is given by $\langle P_O \cup P_U, N_O \cup N_U \rangle$, so the
opens given do form a basis. Notice that $O = \emptyset$ iff $N_O$
is cofinite.

Let $M$ be the full topological model over $T$, and $G$ the
canonical generic: $O \Vdash n \in G$ iff $n \in P_O$. Notice that
nothing can ever be forced out of $G$.

\begin {theorem}
$T \Vdash G$ is unbounded and sequentially bounded.
\end {theorem}

\begin {proof}
It is easy to see that $T \Vdash G$ is unbounded. Let $X \in T$,
and $n \in \mathbb{N}$. Choose $j > n, j \in X$. Let $O$ be given
by $\langle \{j\}, \emptyset \rangle$. Then $X \in O \Vdash j \in
G$. So $T$ is covered by open sets forcing $``\exists i > n \; i
\in G"$, hence $T$ forces the same. Since $n$ was arbitrary, $T
\Vdash ``\forall n \; \exists i > n \; i \in G."$

As for sequential boundedness, let $O \Vdash ``a_n$ is a sequence
through $G"$. We can assume that $O$ is basic and non-empty. We
will need the following fact about the topology: if $P_O'$ is a
finite extension of $P_O$, then the open set $O'$ determined by
$\langle P_O', N_O \rangle$ is compatible with (i.e. is not
disjoint from) every open subset of $O$. Namely, for $V \subseteq
O$ basic open, $P_V \cup (\mathbb{N} - N_V) \in V$; since $V
\subseteq O, \; P_V \cup (\mathbb{N} - N_V) \in O$; that means
$(P_V \cup (\mathbb{N} - N_V)) \cap N_O$ is finite; hence $P_O'
\cup P_V \cup (\mathbb{N} - N_V) \in O' \cap V$.

Returning to the main argument, let $X_O \in O$ be $\mathbb{N} -
N_O$. For each $n$, some basic open neighborhood $U$ of $X_O$
determines the values of $a_n$. $U$ is given by $\langle P_U, N_U
\rangle$, where $P_U$ is a finite subset of $X_O$, and, crucially,
$N_U$ differs from $N_O$ on a finite set, because $U$ is taken to
include $X_O$. Since an open set is unchanged by a finite change
to the second component, we can take $N_U$ to be $N_O$. By the
observation above, any other non-empty basic open subset of $O$ is
compatible with $U$. So any other value of $a_n$ that could be
forced by a subset of $O$ has to be compatible with the one forced
by $U$. So $O$ is covered by opens all forcing the same value of
$a_n$, hence $O$ forces $a_n$ to have that value. Since $O$ also
forces $``a_n \in G"$, and the only numbers $O$ forces to be in
$G$ are those in $P_O$, $O \Vdash a_n \in P_O$. Since $n$ was
arbitrary, $O$ forces $(a_n)$ to be bounded.

\end {proof}

\section{Application: Anti-Specker}
Douglas Bridges has shown \cite {Br} that BD-$\mathbb N$ implies that
the spaces that satisfy a version of the anti-Specker property are
closed under products. He asked whether the reverse implication is
true. Here we show it is not, by showing that in the model of the
previous section the anti-Specker spaces are closed under
products. As stated in the introduction, the purpose of this
exercise is not really anything about anti-Specker itself. Rather,
the goal is to bolster the claim that the models above are the
natural, gentlest models of the failure of BD and BD-$\mathbb N$.
If they are the right models, then they will falsify as little
else that was true in the ground model as possible. Presumably,
all classically true statements that do not imply BD resp.
BD-$\mathbb N$ would remain true in these models. It is entirely
possible that the same result about anti-Specker could be achieved
by consideration of the earlier-known realizability models (see
below, or \cite{Bee, BISV, L}). But nobody has a good guess which
of those models satisfy this anti-Specker closure property and
which not. Furthermore, it might be very difficult to determine
whether the anti-Specker closure holds or not, as opposed to the
argument in the topological model, which admittedly involves some
detail, but is ultimately straightforward. To give another example
of the same situation, it has recently been shown that the Riemann
Permutation Theorem, which is implied by BD-$\mathbb N$, is
strictly weaker than it, via an argument much like the
anti-Specker proof below \cite {LD}. The realizability experts were then
asked whether RPT held in any of the realizability models. They did
not have a clear guess in which it would, and were not able to
prove or refute RPT in them. That such well-qualified experts
could not do this indicates how much easier all of this is with
the topological models.

A metric space $X$ satisfies the one-point anti-Specker property
(notation: $AS^1(X)$) if, for every one-point extension $Z = X
\cup \{*\}$ of $X$ and sequence $(z_n) (n \in \mathbb{N})$ through
$Z$, if $(z_n)$ is eventually bounded away from each point in $X$,
then $(z_n)$ is eventually bounded away from $X$. (The name refers
to Specker's Theorem, which is that in computable mathematics the
closed interval [0,1] does {\em not} have this property \cite {BR,
S}.) Other variants of anti-Specker include $AS(X)$: every
sequence $(z_n)$ through any metric space $Z \supseteq X$
eventually bounded away from each point in $X$ is eventually
bounded away from $X$. $Z$ can also be held fixed (notation:
$AS(X)_Z$). It is known, for instance, that $AS^1([0,1])$ iff
$AS([0,1])_\mathbb R$ \cite {Br09}. For more background on the
anti-Specker properties, see \cite {BeB1, BeB2, Br09, Br0, Br}.

The purpose of this section is achieved with this

\begin {theorem}
$T \Vdash AS^1(X) \wedge AS^1(Y) \rightarrow AS^1(X \times Y).$
\end {theorem}

\begin {proof}
Let $p \Vdash ``AS^1(X) \wedge AS^1(Y) \wedge (z_n)$ is a sequence
through $X \times Y \cup \{*\}$ eventually bounded away from each
$(x,y) \in X \times Y,"$ and $f \in p$. We must find a
neighborhood of $f$ forcing $``(z_n)$ is eventually $*$."

By way of notation, let $x_n$ and $y_n$ be terms such that $p
\Vdash$ ``If $z_n \in X \times Y$ then $z_n = (x_n, y_n)$, and if
$z_n = *$ then $x_n = * = y_n$." Let $I = \sup(rng(f))$. Without
loss of generality, $p$ is basic open, with $g_p({\rm stem}(p))
\geq I$.

\begin{definition} A finite sequence of integers $\sigma$ of length at
least ${\rm stem}(p)$ is {\bf compatible} with $p$ if for all $i <
{\rm stem}(p) \; \sigma(i) = g_p(i)$ and for all $i$ with ${\rm
stem}(p) \leq i < length(\sigma) \; \sigma(i) \leq g_p(i)$. For
$\sigma$ compatible with $p, \; p \upharpoonright \sigma$ is
the open set $q \subseteq p$ such that ${\rm stem}(q) =
length(\sigma),$ for $i < {\rm stem}(q) \; g_q(i) = \sigma(i),$
and otherwise $g_q(i) = g_p(i)$.
\end {definition}

\begin {lemma}There is an open set $q$ such that
$f \in q \subseteq p, {\rm stem}(q) = {\rm stem}(p),$ and for $n
\in \mathbb{N}$ there is a length $i_n$ such that, for all
$\sigma$ of length $i_n$ compatible with $q$, either $q
\upharpoonright \sigma \Vdash ``x_n$ (equivalently $y_n$) = $*$"
or $q \upharpoonright \sigma \Vdash ``x_n$ (equivalently $y_n$) $
\not = *$."
\end {lemma}

\begin {proof}
Let $n \in \mathbb{N}$. First we prove a generalization of this
lemma for this fixed $n$. So let $j \geq {\rm stem}(p), J \geq I$
with $J \leq g_p(j)$, and $\sigma$ be a sequence of length $j$
compatible with $p$. We claim that there is an open set $q$
extending $p$ such that ${\rm stem}(q) = j, g_q \upharpoonright j
= \sigma,$ and $g_q(j) = J$, and there is an $i \geq j$ such that,
for all $\sigma$ of length $i$ compatible with $q$, $q
\upharpoonright \sigma$ decides whether $x_n$ (equiv. $y_n$) is
$*$ or not.

The proof of that claim is similar to that of the lemmas of the
previous section. For notational convenience, extend $p$ if
necessary so that $j = {\rm stem}(p)$. So the claim is that we can
build the desired $q$ only by shrinking $g_p$ beyond ${\rm
stem}(p)$, and even that by not too much (at ${\rm stem}(p)$, we
must still be at least $J$).

Using the notation from the last section, if for each $j \leq J$
$p_j$ had such a good extension $q_j$ with associated integer
$i_j$, then they could be amalgamated to a good extension of $p$,
with the amalgamation of the $i_j$'s being their maximum. So if
$p$ had no good extension, then neither would some direct
extension $p^1 := p_j$ of $p$. Similarly, $p^1$ would itself have
some direct extension $p^2$ with no good extension. Continuing
countably often, the sequence of $p^N$'s determines an $h \in p$.
Some neighborhood $r$ of $h$ must determine whether $x_n$ is $*$
or not, which can be restricted to a good extension of some $p^N$.
This contradicts the choice of $p^N$, so $p$ must have a good
extension.

Now apply the claim with $n = 0, j = {\rm stem}(p),$ and $J = I$,
which determines $\sigma$, to produce $q_0$ and $i_0$. Let $n = 1,
j \geq i_0$ such that $g_{q_0}(j) \geq I+1$, and $J = I+1$. For
each $\sigma$ of length $j$ compatible with $q_0$ and with range
bounded by $I$, use the claim to construct a $q_\sigma$ and an
$i_\sigma$. There are only finitely many such $\sigma$'s, so the
$q_\sigma$'s can be amalgamated via intersection to $q_1$ and the
$i_\sigma$'s via their maximum to $i_1$. More generally, at stage
$k > 0$, let $n = k, j \geq i_{k-1}$ such that $g_{q_{k-1}}(j)
\geq I+k$, and $J = I+k$. Use the claim to construct the
$q_\sigma$'s and $i_\sigma$'s, which are then amalgamated to $q_k$
and $i_k$.

Since the choice of $J$ is unbounded as $k$ runs through the
natural numbers, $q := \bigcap_{k \in \mathbb{N}}q_k$ is an open
set, and has the properties claimed.
\end {proof}

Let $q$ be as in the lemma. The members of $q$ naturally form a
tree $Tr_q$: the nodes are those finite sequences compatible with
$q$, and the members of $q$ are those paths through the tree with
bounded range. At height $j \geq {\rm stem}(q)$ of $Tr_q$, the
amount of splitting is $g_q(j)+1$. The nodes at height $i_n$
determine whether $x_n$ and $y_n$ are $*$ or not. We will have use
for subsets of $q$ the members of which have ranges that are
uniformly bounded. (Such subsets are, of course, not open.) Such
subsets can be given as the set of paths through a subtree $Tr$ of
$Tr_q$ with a fixed bound on the ranges of its nodes, as follows.

\begin {definition} A tree $Tr \subseteq Tr_q$ is {\bf bounded} if
there is a $J$ such that for all $\sigma \in Tr$ and $j <
length(\sigma) \; \sigma(j) < J.$
\end {definition}

\begin {lemma} Let $Tr \subseteq Tr_q$ be bounded. Then $\{ \sigma
\in Tr \mid length(\sigma) = i_n$ and $q \upharpoonright \sigma
\Vdash ``x_n \not = *"\}$ is finite.
\end {lemma}

\begin {proof}
Suppose not. Since $Tr$ (even $Tr_q$, for that matter) is finitely
branching, by K\"onig's Lemma there is a path $h$ such that each
initial segment of $h$ has such an extension. (Note that we do not
claim that $h$ itself goes through infinitely many such nodes,
which would make our lives easier if it were true.) Since $Tr$ is
bounded, $h$ is actually a point in the topological space $T$.

For each natural number $k$, choose an $n_k$ and $\sigma_k$
extending $h \upharpoonright k$ such that $q \upharpoonright
\sigma_k$ forces $``x_{n_k} \not = *"$. We can assume without loss
of generality that for increasing values of $k$ the $n_k$'s are
also increasing and that the $\sigma_k$'s agree with longer
initial segments of $h$. Let $\hat{x}_k$ be a term forced by
$\sigma_k$ to equal $x_{n_k}$ and forced by all other sequences of
the same length as $\sigma_k$ to equal $*$. Since $q \Vdash ``x_i
= *$ iff $y_i = *,"$ we can similarly define $\hat{y}_k$ to be a
term forced by $\sigma_k$ to equal $y_{n_k}$ and by all other
same-lengthed sequences to equal $*$. Let $\hat{z}_n$ be a term
standing for $(\hat{x}_n, \hat{y}_n)$ when those components are
not $*$ and $*$ when they are. Notice that $q \Vdash ``\hat{z}_n$
is either $z_n$ or $*$," so that (for arbitrary $(x,y)$) $q \Vdash
``$If $(z_n)$ is eventually bounded away from $(x,y)$, then so is
$(\hat{z}_n)$."

Since we are assuming that the $\sigma_k$'s agree with $h$ on ever
longer initial segments, then $\sigma_k$ has to extend $h
\upharpoonright k$. Hence if $\sigma(k) \not = h(k)$ then for $j
> k \; q \upharpoonright \sigma \Vdash ``\hat{x}_j = *"$. Hence $q -
\{h\} \Vdash ``(\hat{x}_n)$ is eventually $*$, and so is
eventually bounded away from $X$, and so in particular is bounded
away from each point of $X$." If $q \Vdash ``\forall x \in X \;
(\hat{x}_n)$ is eventually bounded away from $x$", then, since $p
\Vdash AS^1(X)$, $q \Vdash ``(\hat{x}_n)$ is eventually $*$".
This, however, contradicts the choice of $(\hat{x}_n)$ and $h$, as
no neighborhood of $h$ can force that. Hence for some $r \subset
q$ and $x$ we have $r \Vdash ``x \in X,"$ yet $r \not \Vdash
``(\hat{x}_n)$ is eventually bounded away from $x$". Coupled with
the opening observation in this paragraph, no neighborhood of $h$
can force $(\hat{x}_n)$ to be eventually bounded away from $x$. We
would like to thin this sequence so that $x$ is the only point in
$X$ with this property.

Let $r_1 \subseteq q$ force ``$\hat{x}_{n_{k_1}}$ is within 1 of
$x$." By extending $r_1$ if necessary, we can assume that ${\rm
stem}(r_1) \geq length(\sigma_{k_1}).$ In fact, it turns out to be
useful if they are equal. This can be arranged by thinning
$\hat{x}_{j_1}$. That is, letting $\sigma$ be $g_{r_1}
\upharpoonright {\rm stem}(r_1)$, consider a term which is forced
by $q \upharpoonright \sigma$ to be $\hat{x}_{n_{k_1}}$ and,
whenever $\tau \not = \sigma$ has the same length as $\sigma$,
forced by $q \upharpoonright \tau$ to be $*$. By abuse of
notation, we will use the same notation $\hat{x}_{n_{k_1}}$ for
this new term. Furthermore, we want to thin the sequence
$\hat{x}_n$ at places before $n_{k_1}$. So for $k < n_{k_1}$,
change $\hat{x}_k$ if need be to a term standing for $*$.

A desired effect of this thinning is that, for every node $\tau$
incompatible with $\sigma,$ $q \upharpoonright \tau$ forces the
sequence $(\hat{x}_n)$ through (that means including) $n_{k_1}$ to
be the constant $*$. In order to force $\hat{x}_{n_{k_1}}$ to be
within 1 of $x$, though, we have to be working not just beneath
$q$, but also beneath $r_1$. So let $s_1 \subseteq q$ be such that
${\rm stem}(s_1) = {\rm stem}(q)$, for ${\rm stem}(q) \leq k <
{\rm stem}(r_1) \; g_{s_1}(k) = \min(g_q(k), g_{r_1}({\rm
stem}(r_1))),$ and for $k \geq {\rm stem}(r_1) \; g_{s_1}(k) =
g_{r_1}(k).$ To summarize, $s_1 \Vdash ``$ for $k \leq j_1$ either
$\hat{x}_k = *$ or $\hat{x}_k$ is within 1 of $x$," and ${\rm
stem}(s_1) = {\rm stem}(q)$.

Let $r_2 \subseteq s_1$ force ``$\hat{x}_{n_{k_2}}$ is within 1/2
of $x$." By extending $r_2$ if necessary, we can assume that
$g_{s_1}({\rm stem}(r_2)-1) \geq I+1.$ Thin $(\hat{x}_n)$
similarly to the above, and define $s_2 \subseteq s_1$ to be such
that ${\rm stem}(s_2) = {\rm stem}(q)$, for $k < {\rm stem}(r_1)
\; g_{s_2}(k) = g_{s_1}(k)$, for ${\rm stem}(r_1) \leq k < {\rm
stem}(r_2) \; g_{s_2}(k) = \min(g_{s_1}(k), g_{r_2}({\rm
stem}(r_2))),$ and for $k \geq {\rm stem}(r_2) \; g_{s_2}(k) =
g_{r_2}(k).$ What this gets us is that the stem is not increasing,
for $k$ between $n_{k_1}$ and $n_{k_2}$ $\hat{x}_k$ is forced to
be either $*$ or within 1/2 of $x$, and $g_{s_2}({\rm
stem}(r_2)-1) \geq I+1$. That last fact will be preserved at
future steps, enabling us to take the intersection of these open
sets at the end and still have an open set.

Inductively, let $r_{i+1} \subseteq s_i$ force
$\hat{x}_{n_{k_{i+1}}}$ is within $1/i$ of $x$." By extending
$r_{i+1}$ if necessary, assume $g_{s_i}({\rm stem}(r_{i+1})-1)
\geq I+i.$ Thin $(\hat{x}_n)$, and define $s_{i+1} \subseteq s_i$
to be such that ${\rm stem}(s_{i+1}) = {\rm stem}(q)$, for $k <
{\rm stem}(r_i) \; g_{s_{i+1}}(k) = g_{s_i}(k)$, for ${\rm
stem}(r_i) \leq k < {\rm stem}(r_{i+1}) \; g_{s_{i+1}}(k) =
\min(g_{s_i}(k), g_{r_{i+1}}({\rm stem}(r_{i+1}))),$ and for $k
\geq {\rm stem}(r_{i+1}) \; g_{s_{i+1}}(k) = g_{r_{i+1}}(k).$ As
before, the stem is not increasing, for $k$ between $n_{k_i}$ and
$n_{k_{i+1}}$ $\hat{x}_k$ is forced to be either $*$ or within
$1/i$ of $x$, and $g_{s_{i+1}}({\rm stem}(r_{i+1})-1) \geq I+i$.

Finally, let $s_\infty = \bigcap_i s_i$. We have $h \in s_\infty,
s_\infty$ is open, and for all $t \subseteq s_\infty,$ if $h \in t
\Vdash ``(\tilde{x}_n)$ is a subsequence of $(\hat{x}_n)$," then
$t \not \Vdash ``(\tilde{x}_n)$ is eventually bounded away from
$x$."

Just as we had earlier found $r \subseteq q$ and $x$ with $r
\Vdash ``x \in X"$ yet $r \not \Vdash ``(\hat{x}_n)$ is eventually
bounded away from $x$," there are $t \subseteq s_\infty$ and $y$
such that $t \Vdash ``y \in Y,"$ yet $t \not \Vdash ``(\hat{y}_n)$
is eventually bounded away from $y$." Because of the monotonicity
of $(\hat{x}_n)$ approaching $x$ forced by $s_\infty$, $t \not
\Vdash ``(\hat{z}_n)$ is eventually bounded away from $(x,y)$."
Hence $t \not \Vdash ``(z_n)$ is eventually bounded away from
$(x,y)$." But this contradicts the opening hypothesis of the whole
theorem.
\end {proof}

Armed with this lemma, we are almost done. Consider the sub-tree
$Tr_1$ of $Tr_q$ of all finite sequences with entries less than or
equal to $I+1$. By the lemma, $\{ \sigma \in Tr_1 \mid
length(\sigma) = i_n$ and $q \upharpoonright \sigma \Vdash ``x_n
\not = *"\}$ is finite. Let $j_{I+1}$ be the maximum of the
lengths of the nodes in that set. So if some node of greater
length forces some $x_n$ not to be $*$, that node must have an
entry larger than $I+1$. More particularly, there is a largest
natural number, say $M$, such that $i_M \leq j_{I+1}$. For $m>M$,
if a node forces $x_m$ not to be $*$, then that node has an entry
greater than $I+1$. We start to define a function $g$. Let $g
\upharpoonright {\rm stem}(q) = g_q \upharpoonright {\rm
stem}(q)$, and for ${\rm stem}(q) \leq k \leq j_{I+1}, \; g(k) =
I$.

Now consider the larger sub-tree that on all levels $\leq j_{I+1}$
has only entries $\leq I$ (so is compatible with $g$), and beyond
that all numbers $\leq I+2$ may appear. (That is, $\sigma$ is in
the sub-tree iff $\sigma(k) \leq I$ for $k \leq j_{I+1}$ and
$\sigma(k) \leq I+2$ otherwise.) Again, since there are only
finitely many nodes forcing some $x_n$ not to be $*$, let
$j_{I+2}$ be the maximum of their lengths. Extend $g$ so that for
$j_{I+1} < k \leq j_{I+2}, \; g(k) = I+1$. Notice that, when we
use $g$ as the bounding sequence $g_r$ of a basic open set $r$,
there will be no nodes allowed by $g$ of length between $j_{I+1}$
and $j_{I+2}$ allowing an $x_n$ not to be $*$.

In general, at stage $e$, consider the sub-tree with growth
controlled up to height $j_{I+e}$ by the amount of $g$ built so
far, and allowing entries up to $I+e+1$ after that. Let
$j_{I+e+1}$ bound the lengths of nodes which force some $x_n$ not
to be $*$. Extend $g$ to be defined up to $j_{I+e+1}$ with the new
values being $I+e$.

After countably many of these steps, we will have defined $g$ to
be total. Let the basic open set $r$ be such that ${\rm stem}(r) =
{\rm stem}(q)$ and $g_r = g$. Then $f \in r \Vdash ``\forall m > M
\; x_m = y_m = *;"$ in other words,  $f \in r \Vdash ``(z_n)$ is
eventually $*$."
\end {proof}

\section {Realizability models}

The first model above was the first developed with the intention
of falsifying BD-$\mathbb{N}$. It was not the first observed to
falsify BD-$\mathbb{N}$. Namely, Ishihara \cite {I91,I92} showed
that certain continuity principles are equivalent with certain
foundational constructive principles, among which is
BD-$\mathbb{N}$. Continuity was studied well before
BD-$\mathbb{N}$ was ever identified, and models, apparently all of
them realizability models, were developed in which these
continuity properties fail. It was later observed that the other
foundational principles identified by Ishihara hold in these
models, and then concluded that BD-$\mathbb{N}$ must fail. This is
all very true, but somewhat unsatisfying. The argument is
roundabout. One would naturally ask, for instance, just what is
the pseudo-bounded yet unbounded set. The answer is, of course,
implicit in the chain of arguments leading to the conclusion that
BD-$\mathbb{N}$ fails in these models. It just takes some work
digging through all of that. In this section, we do that work for
a representative (albeit not random) sampling of these models.

\subsection {Extensional realizability}

In \cite {L} Lietz provides a thorough overview of realizability
models, and an analysis of the continuity principles validated and
falsified in some particularly interesting ones. We will examine
only one of these, extensional realizability ({\bf Ext}). To make
the paper self-contained, we will give the basics of {\bf Ext};
for more background, see \cite {O} or \cite {Bees}, ch. XI sec.
20.\footnote {Thanks are due here to Thomas Streicher for his
correspondence explaining {\bf Ext} to me.}

The objects are partial equivalence relations on the natural
numbers, which are also viewed as codes for computable functions
(in some standard way) when considering application. On the bottom
level, the naturals themselves, extensional equality is just
equality. A function from $\mathbb N$ to $\mathbb N$, i.e. a
member of $\mathbb N^\mathbb N$, is given by an index $e$ of a
total computable function. If two such indices, say $e$ and $e'$,
yield the same functions, then they are extensionally equal. For
an index $i$ to stand for a function from $\mathbb N^\mathbb N$ to
$\mathbb N$, on equal inputs $i$ must yield equal outputs:
$\{i\}(e) = \{i\}(e')$. (For $i$ to be a function from $\mathbb
N^\mathbb N$ to $\mathbb N^\mathbb N$, the outputs on
extensionally equal inputs do not have to be numerically equal,
just extensionally equal.)

At the level of the naturals, extensionality plays no role, and we
have the following fact, true also in many other realizability
models.

\begin {proposition}
In {\bf Ext}, every function from $\mathbb{N}$ to $\mathbb{N}$ is
computable.
\end {proposition}

\begin {proof}
Let $e \Vdash ``f$ is a function from $\mathbb{N}$ to
$\mathbb{N}$." So $e \Vdash ``\forall n \; \exists m \; f(n)=m."$
Hence $\forall n \; (\{e\}(n))_1 \Vdash f(n) = \{e\}(n)_0$. So $f
= \lambda n.\{e\}(n)_0$ is computable.
\end {proof}

In fact, that proposition is almost enough to get BD-$\mathbb{N}$
to be true! Let $A$ be any countable set of naturals, and $f$ any
counting of $A$. Assuming a classical meta-theory, either $A$ is
bounded or it's not. If it is, great. If not, let $a_0$ be $f(0)$
and $a_{n+1}$ be the first value of $f$ greater than $a_n$.
$(a_n)$ is computable, and witnesses that $A$ is not
pseudo-bounded.

So, in contrast to the topological models, there is no specific
counter-example. Does that mean that BD-$\mathbb{N}$ is true?

\begin {theorem}
(\cite {L}) In {\bf Ext}, BD-$\mathbb{N}$ is false.
\end {theorem}

What's at stake is uniformity. Each instance of BD-$\mathbb{N}$ is
true, just not uniformly so.

Central to this proof is the KLST Theorem \cite{KST,Ts}. As should
become clear, BD-$\mathbb N$, or the lack thereof, could be viewed
as the difference KLST being true in the classical meta-theory and
being true internally in {\bf Ext}. It could also be viewed as the
gap within {\bf Ext} between the full KLST and the fragment of
KLST that happens to be true there, sequential continuity. KLST is
the following result in classical computability (then called
recursion) theory:

\begin {theorem}
(KLST) Every computable, integer-valued function, with domain
including the indices of the total computable functions, which is
extensional on those indices, is continuous. That is, there is a
partial computable function $M$ such that, if $\{z\}(y)$ converges
for every $y$ with $\{y\}$ total, and $\{z\}(y) = \{z\}(y')$
whenever $\{y\}$ and $\{y'\}$ are total and equal to each other,
then, for $\{y\}$ total, $M(z,y)$ is a modulus of convergence for
$\{z\}(y)$.
\end{theorem}

Lietz \cite L used KLST to show the sequential continuity of all
functions from $\mathbb N^\mathbb N$ to $\mathbb N$ in {\bf Ext.}
Ishihara's \cite {I92} analysis of continuity has as a particular
case that BD-$\mathbb N$ yields that sequential continuity implies
continuity. Troelstra \cite T showed that extensionality plus a
modest amount of choice (which holds in {\bf Ext}) implies that
not all functions from $\mathbb N^\mathbb N$ to $\mathbb N$ are
continuous; a more accessible source is \cite {Bees}, ch. XI sec.
19. Taken together, as observed in \cite L, {\bf Ext} falsifies
BD-$\mathbb N$. The following argument takes those three proofs,
applies them to {\bf Ext}, and pulls out the concrete
counter-example to BD-$\mathbb N$ in {\bf Ext}.

\begin {proof}
First, in {\bf Ext}, every $F:\mathbb N^\mathbb N \rightarrow
\mathbb N$ is sequentially continuous, as follows. Suppose $g_n
\rightarrow g$ in $\mathbb N^\mathbb N$. A realizer $z \Vdash
F:\mathbb N^\mathbb N \rightarrow \mathbb N$ is also an index for
computing $F$: $\{z\}(x) = F(\{x\})$. Similarly, a realizer $y
\Vdash g \in \mathbb N^\mathbb N$ computes $g$: $\{y\}=g$. By
KLST, $M(z,y)$ is a modulus of convergence for $F$ at $g$.
However, $M$ may not be the index of a function of higher type in
{\bf Ext}, as it may not be extensional. We seek something more
modest -- a modulus of convergence only for the sequence $(g_n)$
-- yet this modulus must be extensional. Since $(g_n)$ converges
to $g$, there is a $k$ beyond which $g_n\upharpoonright M(z,y) =
g\upharpoonright M(z,y)$. That is, for $n>k$, $g_n$ agrees with
$g$ up to a modulus of convergence, so $F(g_n) = F(g)$. So
evaluate $\{z\}(y_n)$ for $n$ from 0 through $k$, and pick the
least $n$ beyond which $\{z\}(y_n)$ is the constant value $F(g)$.
That point witnesses the sequential continuity of $F$ for $g_n
\rightarrow g$.

By way of notation, let $e_0$ be a canonical index for the
constant 0 function: $\{e_0\}(n) = 0$. By saying $g$ extends
$0^m$, we mean that for $x<m \; g(x)=0$. Later on we will have use
for the type 2 version of $e_0$, which we call $E_0$, the constant
0 function with inputs from $\mathbb N^\mathbb N$.

Working in {\bf Ext}, let $\{z\} = F:\mathbb N^\mathbb N
\rightarrow \mathbb N$. Let $A_z$ be $\{0\} \cup \{m \mid$ there
is a $g \in \mathbb N \rightarrow \mathbb N$ extending $0^m$ and
eventually 0 with $\{z\}(g) \not = \{z\}(e_0) \}$. $A_z$ is
countable: there are only countably many eventually 0 $g$'s, say
$(g_n)$; at stage $i$ evaluate $\{z\}(g_i)$; if that's unequal to
$\{z\}(e_0)$ then generate the appropriate integers into $A_z$,
else generate another 0. Let $e_z$ be an index for a counting of
$A_z$.

Moreover, $A_z$ is pseudo-bounded, as follows. Given a sequence
$(m_n)$ of members of $A_z$, let $h_n$ be $\{e_0\}$ if $m_n < n$,
and the least corresponding $g$ (i.e. extending $0^{m_n}$ and
eventually 0 with $\{z\}(g) \not = \{z\}(e_0)$) if $m_n \geq n$.
Since $h_n$ extends $0^n$, $h_n \rightarrow \{e_0\}$. By
sequential continuity, we have an index beyond which $\{z\}(h_n) =
\{z\}(e_0)$. Whenever $\{z\}(h_n) = \{z\}(e_0)$, we cannot be in
the second case in the definition of $h_n$. So we're in the first
case: $m_n < n$. This is exactly the pseudo-boundedness of $A_z$.
Let $f_z$ be the realizer for the pseudo-boundedness of $A_z$ just
constructed.

To show that BD-$\mathbb N$ is not realized, it is enough to
suppose it is, and come up with a contradiction. So suppose $b
\Vdash ``$if $A \subseteq \mathbb N$ is countable and
pseudo-bounded then $A$ is bounded." In particular, if $z$ is an
index as above, $\{b\}(e_z,f_z) \Vdash ``A_z$ is bounded," and
$\{b\}(e_z,f_z)_0$ is a bound for $A_z$. Let $m$ be
$\{b\}(e_{E_0},f_{E_0})_0$.

Given $\beta:\mathbb N \rightarrow \mathbb N$, let $F_\beta :
\mathbb N ^ \mathbb N \rightarrow \mathbb N$ with index $z_\beta$
be as follows. Given $\alpha \in \mathbb N ^ \mathbb N$,
$F_\beta(\alpha)$ depends only on $\alpha(m+1)$. If $\alpha(m+1) =
0$, then $F_\beta(\alpha) = 0;$ else $F_\beta(\alpha) =
\beta(\alpha(m+1)-1).$ In words, to see $F_\beta$, take the
countably branching tree $\mathbb N^{< \mathbb N}$; go up to the
$m^{th}$ level; each node there has countably many immediate
successors; label the 0$^{th}$ successor with 0, and spread
$\beta$ out on the other successors; given $\alpha$ a branch
through that tree, follow $\alpha$ up to level $m+1$ and return
the value encountered there.

If $\beta = \{e_0\}$, then $F_\beta = \{E_0\}$, and by
extensionality, $\{b\}(e_{z_\beta},f_{z_\beta})_0 = m$. On the
other hand, if $\beta \not = \{e_0\},$ then $A_{z_\beta} = \{0,
... , m+1\}$. Hence $m$ would not be a bound for $A_{z_\beta}$,
and $\{b\}(e_{z_\beta},f_{z_\beta})_0 > m$.

To conclude, $\lambda \beta . \{b\}(e_\beta,f_\beta)_0$ is a
total, computable, extensional function. By KLST, it's continuous.
But we've just seen it's not: at $\{e_0\}$ it returns $m$, but
does not do so in any neighborhood of $\{e_0\}$.
\end {proof}

\subsection {fp-realizability}

Beeson \cite{Bee} introduced {\bf formal-provable realizability},
abbreviated {\bf fp-realizability}, in order to show the
independence from a theory of constructive arithmetic of some
continuity theorems, namely KLST (there called KLS), discussed in
the previous section, and MS (Myhill-Shepherdson), the variant of
KLST for partial computable functions. Beeson and Scedrov
\cite{BS} extended fp-realizability to a model of full IZF set
theory. Much later, following Ishihara's analysis of continuity,
Bridges et al. \cite{BISV} realized that fp-realizability
validates the other principles Ishihara identified, and so must
falsify BD-$\mathbb{N}$. In this section, we bring out exactly how
BD-$\mathbb{N}$ fails.

For our purposes, the most important clause in the inductive
definition of realizability is implication. In standard
realizability, this is given as:
\begin {center}
$e \Vdash \phi \rightarrow \psi$ iff $\forall x \; (x \Vdash \phi
\rightarrow \{e\}(x) \Vdash \psi)$.
\end {center}
Kleene also defined a modification of this realizability, that
includes not only that $x$ must realize $\phi$ but also that
$\phi$ must be provable:
\begin {center}
$e \Vdash \phi \rightarrow \psi$ iff $\forall x \; (x \Vdash \phi
\wedge Pr(\phi) \rightarrow \{e\}(x) \Vdash \psi)$,
\end {center}
where $Pr$ is some appropriate proof predicate. Beeson's
fp-realizability does this one step better, by having that not
only must $x$ realize $\phi$, and not only that $\phi$ is
provable, but even that the realizability of $\phi$ by $x$ must be
provable:
\begin {center}
$e \Vdash \phi \rightarrow \psi$ iff $\forall x \; (Pr(x \Vdash
\phi) \rightarrow \{e\}(x) \Vdash \psi)$,
\end {center}
where we can afford to eliminate the clause $x \Vdash \phi$ from
the antecedent by the presumed soundness of the provability
predicate. What makes this work for our purposes is that more
needs to be put in than needs to be put out: the input must be
provably realizing, the output merely realizing.

The version of fp-realizability Beeson uses is non-numerical, in
that the realizers themselves are suppressed. That is, a
translation is given from formulas $\phi$ to formulas $\phi^r$,
where the latter should be thought of as ``$\phi$ is realized."
The reason given for doing that is that it makes the proof of
soundness of fp-realizability easier, although it is observed that
it actually makes a difference in some cases about what's
realized, fortunately not in any cases of current interest. The
version of fp-realizability given below contains the realizers,
because the reader is likely to be more familiar and comfortable
with such a presentation. It was derived from Beeson's
non-numerical variant by making what seemed like the only possible
such extrapolation. The clauses are:
\begin {flushleft}
$e \Vdash \phi$ iff $\phi$ (for $\phi$ atomic)

$e \Vdash \phi \wedge \psi$ iff $e_0 \Vdash \phi$ and $e_1 \Vdash
\psi$

$e \Vdash \phi \vee \psi$ iff ($e_0 = 0$ and $Pr(e_1 \Vdash
\phi)$) or ($e_0 = 1$ and $Pr(e_1 \Vdash \psi)$

$e \Vdash \phi \rightarrow \psi$ iff for all $x$ if $Pr(x \Vdash
\phi)$ then $\{e\}(x) \Vdash \psi$

$e \Vdash \forall x \; \phi(x)$ iff for all $x \; e \Vdash
\phi(x)$

$e \Vdash \exists x \; \phi(x)$ iff $Pr(e_1 \Vdash \phi(e_0))$.
\end {flushleft}
As usual, $\neg \phi$ is an abbreviation for $\phi \rightarrow
0=1$.

\begin {theorem}
(\cite {BISV}) Under fp-realizability, BD-$\mathbb{N}$ is false.
\end {theorem}

By analogy with the realizability from the previous section, and
realizability in general, you might think that the failure of
BD-$\mathbb{N}$ is once again the lack of uniformity. Quite to the
contrary, here we have a case of a particular counter-example
instead.

\begin {definition}
$\{w\}(z) \downarrow_{<n}$ if the function coded by $w$ when
applied to $z$ converges in fewer than $n$ many steps with output
less than $n$.
\end {definition}

\begin {definition}
Let $\{v\}(n) = max \{k<n \mid \forall j,w,z < k$ if $j$ codes a
proof that $\{w\}$ is total then $\{w\}(z) \downarrow_{<n}$\}.
\end {definition}

\begin {proof}
We will show that rng($\{v\}$) is the desired counter-example.

For the countability of rng($\{v\}$), we need that $\{v\}$ is
realized to be total. The realizer for this is $v$ itself, which
works as long as $\{v\}$ is actually total. That's the case
because $\{v\}(n)$ is the maximum of a bounded set, and membership
in the set is determined by a finite search.

Now consider the task of realizing that the range of $\{v\}$ is
not bounded. First note that $\{v\}$ is actually unbounded: to get
an $n_k$ with $\{v\}(n_k) \geq k$, we need to consider all proofs
$j<k$ that $w<k$ codes a total function; by soundness $\{w\}$ then
is total; so just wait long enough so that for all such $w$ and
$z<k \; \{w\}(z)$ has converged. So no $e$ could realize that $k$
is a bound to rng($\{v\}$), because $n_k$ is a counter-example to
that. Hence nothing can realize that rng($\{v\}$) is bounded. By
the definition of forcing a negation, everything realizes that
rng($\{v\}$) is not bounded.

The real work is showing that the range of $\{v\}$ is
pseudo-bounded. We need to realize ``if $f$ enumerates a subset of
$rng\{v\}$ then there is a bound beyond which $f(n) \leq n$."
Suppose $x$ provably realizes the antecedent: \begin {center}
$Pr(x \Vdash \forall i \; \exists m \; f(i) = \{v\}(m))$.
\end {center}
In particular, for all $i, \; f(i) = \{v\}(\{x\}(i)_0)$.

Let $N>x$ code such a proof. In particular, $N$ also proves that
$\{x\}$ is total, which is all we need. Then for $n>N$

\qquad $f(n) = \{v\}(\{x\}(n)_0)$

\qquad \qquad \;= max $\{k<\{x\}(n)_0 \mid \forall j,w,z < k$ if
$j$ codes a proof

\qquad \qquad \qquad \qquad \qquad \qquad \qquad that $w$ is total
then $\{w\}(z) \downarrow <_{\{x\}(n)_0}$\}.\\ Consider any $k >
n$, by way of seeing whether it's in the set above. Let $j, w, z$
be $N, x, n$, respectively. We need to consider whether
$\{x\}(n)\downarrow <_{\{x\}(n)_0}$. That could not happen, since
$\{x\}(n) \geq \{x\}(n)_0$, as a pair is at least as large as each
of its components. So $f(n)$ is the max of a set which includes
nothing greater than $n$, hence $f(n) \leq n$.
\end {proof}

\section{Questions}
Both topological models presented here were called the natural
models. This was done because it feels right. They seem like the
obvious guesses for topological models violating BD and
BD-$\mathbb N$. Also, they seem to violate as little else as
possible, as for instance the property of anti-Specker spaces
discussed still holds. But what could it mean for these models to
be natural? How could that be made more precise?

What other independence results could these models show? What
other consequences of BD and BD-$\mathbb N$ might still hold in
them?

We discussed two realizability models, one with a specific
counter-example, the other with no counter-example, just a lack of
uniformity. The topological models both have counter-examples. Is
there a topological failure of BD or BD-$\mathbb N$ with no one
counter-example? This is possible on general principles: it could
be that every pseudo-bounded set is not not bounded, while there
is no open set forcing a bound for each pseudo-bounded set
simultaneously. A general way of doing this is forcing with
settling \cite {LR, Lu}. The reason that answer is not
satisfactory is that settling does not produce a model of IZF.
Power Set must fail; even Subset Collection would. So what would
be a topological model of IZF in which each instance of BD (resp.
BD-$\mathbb N$) holds densely but BD (resp. BD-$\mathbb N$)
doesn't?

We discussed the models over the space of bounded sequences and
the space of unbounded sets. What are the models like over the
space of unbounded sequences and the space of bounded sets? Is
there anything interesting going on there, especially relative to
BD and BD-$\mathbb N$?


\begin {thebibliography} {99}
\bibitem{Ba} James E. Baumgartner, ``Iterated forcing," in {\bf
Surveys in Set Theory}, London Mathematical Society Lecture Note
Series, vol. 87, Cambridge University Press, Cambridge, 1983, p.
1-59
\bibitem{Bee} Michael Beeson, ``The nonderivability in intuitionistic formal
systems of theorems on the continuity of effective operations,"
{\bf Journal of Symbolic Logic}, v. 40 (1975), p. 321-346
\bibitem{Bees} Michael Beeson, {\bf Foundations of Constructive
Mathematics}, Springer, Berlin, 1985
\bibitem{BS} Michael Beeson and  Andre Scedrov,  ``Church's thesis,
continuity, and set theory," {\bf Journal of Symbolic Logic}, v.
49 (1984), p. 630-643
\bibitem{BeB1} Josef Berger and Douglas Bridges, ``A Fan-theoretic
equivalent of the antithesis of Specker's Theorem," {\bf Proc.
Koninklijke Nederlandse Akad. Wetenschappen} (Indag. Math., N.S.),
v. 18 (2007), p. 195-202
\bibitem{BeB2} Josef Berger and Douglas Bridges, ``The anti-Specker
property, a Heine–Borel property, and uniform continuity," {\bf
Archive for Mathematical Logic}, v. 46 (2008), p. 583-592
\bibitem{B} Errett Bishop, {\bf Foundations of Constructive
Mathematics}, McGraw-Hill, New York, 1967
\bibitem{BB} Errett Bishop and Douglas Bridges, {\bf Constructive
Analysis}, Springer, Berlin, 1985
\bibitem{Br09} Douglas Bridges, ``Constructive notions of
equicontinuity," {\bf Archive for Mathematical Logic}, v. 48
(2009), p. 437-448
\bibitem{Br0} Douglas Bridges, ``Uniform equicontinuity and the
antithesis of Specker's Theorem," unpublished
\bibitem{Br} Douglas Bridges, ``Inheriting the anti-Specker property",
preprint, University of Canterbury, NewZealand, 2009, submitted
for publication
\bibitem{BISV} Douglas Bridges, Hajime Ishihara, Peter Schuster,
and Luminita Vita, ``Strong continuity implies uniformly
sequential continuity," {\bf Archive for Mathematical Logic}, v.
44 (2005), p. 887-895
\bibitem{BR} Douglas Bridges and Fred Richman, {\bf Varieties of
Constructive Mathematics}, London Mathematical Society Lecture
Note Series, vol. 97, Cambridge University Press, Cambridge, 1987
\bibitem{G1} Robin J. Grayson, ``Heyting-valued models for intuitionistic
set theory," in {\bf Applications of Sheaves}, Lecture Notes in
Mathematics, vol. 753 (eds. Fourman, Mulvey, Scott), Springer,
Berlin Heidelberg New York, 1979, p. 402-414
\bibitem{G2} Robin J. Grayson, ``Heyting-valued semantics," in {\bf Logic
Colloquium '82}, Studies in Logic and the Foundations of
Mathematics, vol. 112 (eds. Lolli, Longo, Marcja) , North-Holland,
Amsterdam New York Oxford, 1984, p. 181-208
\bibitem{I91} Hajime Ishihara, ``Continuity and nondiscontinuity
in constructive mathematics," {\bf Journal of Symbolic Logic}, v.
56 (1991), p. 1349-1354
\bibitem{I92} Hajime Ishihara, ``Continuity properties in
constructive mathematics," {\bf Journal of Symbolic Logic}, v. 57
(1992), p. 557-565
\bibitem{I01} Hajime Ishihara, ``Sequential continuity in
constructive mathematics," in {\bf Combinatorics, Computability,
and Logic} (eds. Calude, Dinneen, and Sburlan), Springer, London,
2001, p. 5-12
\bibitem{IS} Hajime Ishihara and Peter Schuster, ``A Continuity
principle, a version of Baire's Theorem and a boundedness
principle," {\bf Journal of Symbolic Logic}, v. 73 (2008), p.
1354-1360
\bibitem{IY} Hajime Ishihara and Satoru Yoshida, ``A Constructive
look at the completeness of {\it D}({\bf R})," {\bf Journal of
Symbolic Logic}, v. 67 (2002), p. 1511-1519
\bibitem{KST} Georg Kreisel, Daniel Lacombe, and Joseph Shoenfield,
``Partial recursive functions and effective operations," in {\bf
Constructivity in Mathematics} (ed. Arend Heyting), North-Holland,
1959, p. 195-207
\bibitem{L} Peter Lietz, ``From Constructive Mathematics to
Computable Analysis via the Realizability Interpretation," Ph.D.
thesis, Technische Universit\"at Darmstadt, 2004,
http://www.mathematik.tu-darmstadt.de/~streicher/THESES/lietz.pdf.gz
\bibitem{Lu} Robert Lubarsky, ``Topological forcing semantics with
settling," in {\bf Proceedings of LFCS '09, Lecture Notes in
Computer Science No.5407} (eds. Sergei N. Artemov and Anil
Nerode), Springer, 2009, p. 309-322; also {\bf Annals of Pure and
Applied Logic}, to appear, doi: 10.1016/j.apal.2011.09.014
\bibitem {RSL10} Robert Lubarsky, ``Geometric Spaces with No
Points," {\bf Journal of Logic and Analysis}, v. 2 No. 6 (2010),
p. 1-10, http://logicandanalysis.org/, doi: 10.4115/jla2010.2.6
\bibitem{LD} Robert Lubarsky and Hannes Diener, ``Principles
Weaker than BD-$\mathbb N$," submitted for publication
\bibitem{LR} Robert Lubarsky and Michael Rathjen, ``On the constructive
Dedekind Reals," {\bf Logic and Analysis}, v. 1 (2008), pp.
131-152; also in {\bf Proceedings of LFCS '07, Lecture Notes in
Computer Science No. 4514} (eds. Sergei N. Artemov and Anil
Nerode), Springer, 2007, pp. 349-362
\bibitem{O} Jap van Oosten, ``Extensional realizability," {\bf Annals
of Pure and Applied Logic}, v. 84 (1997), p. 317-349
\bibitem{S} Ernst Specker, ``Nicht konstruktiv beweisbare S\"atze der
Analysis," {\bf Journal of Symbolic Logic}, v. 14 (1949), p.
145-158
\bibitem{T} Anne S. Troelstra, ``A Note on non-extensional
operations in connection with continuity and recursiveness," {\bf
Indagationes Mathematicae}, v. 39 (1977), p. 455-462
\bibitem{TvD} Anne S. Troelstra and Dirk van Dalen, {\bf
Constructivism in Mathematics}, vol. 1, North-Holland, Amsterdam,
1988
\bibitem{Ts} G.S. Tseitin, ``Algorithmic operators in constructive
complete metric spaces," {\bf Doklady Akademii Nauk SSSR}, v. 128
(1959), p. 49-52
\end {thebibliography}
\end{document}